\documentclass[smallextended]{svjour3}
\usepackage{latexsym,amssymb,amsmath}
\smartqed
\usepackage{graphicx}
\usepackage{xcolor}
\RequirePackage{fix-cm}
\usepackage[section]{placeins}
\def\F{Fr\'{e}chet}
\begin{document}

\title{Subdifferential Stability Analysis for Convex Optimization Problems via Multiplier Sets\footnote{Dedicated to Professor Michel Th\'era on the occasion of his seventieth birthday.}}

\author{D.T.V.~An        \and
        N.D.~Yen 
}

\institute{D.T.V.~An \at
              Department of Mathematics and Informatics, Thai Nguyen University of Sciences, Thai Nguyen city, Vietnam\\
                            \email{andtv@tnus.edu.vn.}                      \and
           N.D. Yen \at
              Institute
              of Mathematics, Vietnam Academy of Science and Technology, 18 Hoang
              Quoc Viet, Hanoi 10307, Vietnam\\
              \email{ndyen@math.ac.vn.}
}

\date{Received: date / Accepted: date}

\maketitle

\begin{abstract} This paper discusses differential stability of convex programming problems in Hausdorff locally convex  topological vector spaces. Among other things, we obtain formulas for computing or estimating the subdifferential and the singular subdifferential of the optimal value function via suitable multiplier sets.


\end{abstract}

\section{Introduction}
\markboth{\centerline{\it Introduction}}{\centerline{\it D.T.V.~An
		and N.D.~Yen}} \setcounter{equation}{0}

Investigations on differentiability properties of the \textit{optimal value function} and of the \textit{solution map} in parametric mathematical programming are usually classified as studies on differential stability of optimization problems. Some results in this direction can be found in  \cite{AnYao,AnYen,Aubin_1998,Auslender_1979,Bonnans_Shapiro_2000,Dien_Yen_1991,Gauvin_Dubeau_1982,Gauvin_Dubeau_1984,Gauvin_Tolle_1977,Gollan_1984,Mordukhovich_2006,MordukhovichEtAl_2009,Rockafellar_1982,Thibault_1991}, and the references therein. 

 For differentiable nonconvex programs, the works of Gauvin and Tolle \cite{Gauvin_Tolle_1977}, Gauvin and Dubeau \cite{Gauvin_Dubeau_1982}, and Lempio and Maurer \cite{Lempio_Maurer_1980} have had great impacts subsequently. The authors of the first two papers studied parametric programs in a finite-dimensional setting, while a Banach space setting was adopted in the third one. The main ideas of those papers are to use linear linearizations and a regularity condition (either the Mangasarian-Fromovitz Constraint Qualification, or the Robinson regularity condition). Formulas for computing or estimating the Dini directional derivatives, the classical directional derivative, or the Clarke generalized directional derivative and the Clarke generalized gradient of the optimal value function, when the problem data undergoes smooth perturbations, were given in \cite{Gauvin_Dubeau_1982,Gauvin_Tolle_1977,Lempio_Maurer_1980}. Gollan \cite{Gollan_1984}, Outrata~\cite{Outrata}, Penot~\cite{Penot1997}, Rockafellar \cite{Rockafellar_1982}, Thibault \cite{Thibault_1991}, and many other authors, have shown that similar results can be obtained for nondifferentiable nonconvex programs. In particular, the links of the subdifferential of the optimal value function in the contingent sense and in the \F \ sense with multipliers were pointed in~\cite{Penot1997}. Note also that, if the objective function is nonsmooth and the constraint set is described by a set-valued map, differential stability analysis can be investigated by the primal-space approach; see~\cite{Outrata} and the references therein. 
 
 For optimization problems with inclusion constraints in Banach spaces, differentiability properties of the optimal value function have been established via the dual-space approach by Mordukhovich \textit{et al.} in \cite{MordukhovichEtAl_2009}, where it is shown that the new general results imply several fundamental results which were obtained by the primal-space approach.

 Differential stability for convex programs has been studied intensively in the last five decades. A formula for computing the subdifferential of the optimal value function of a standard convex mathematical programming problem with right-hand-side perturbations, called the \textit{perturbation function}, via the set of the \textit{Kuhn-Tucker vectors} (i.e., the vectors of Kuhn-Tucker coefficients; see \cite[p.~274]{Rockafellar_1970}) was given by Rockafellar \cite[Theorem~29.1]{Rockafellar_1970}. Until now, many analogues and extensions of this classical result have been given in the literature.
   
 New results on the exact subdifferential calculation for optimal value functions involving coderivatives of constraint set mapping have been recently obtained by Mordukhovich \textit{et al.} \cite{Mordukhovich_Nam_Rector_Tran} for optimization problems in Hausdorff locally convex topological vector spaces, whose convex marginal functions are generated by arbitrary convex-graph multifunctions. Actually, these developments extend those started by Mordukhovich and Nam \cite[Sect.~2.6]{Mordukhovich_Nam2014} and \cite{Mordukhovich_Nam2015} in finite dimensions. 
 
 Recently, by using the Moreau-Rockafellar theorem and appropriate regularity conditions, An and Yao \cite{AnYao}, An and Yen \cite{AnYen} have obtained formulas for computing the subdifferential and the singular subdifferential of the optimal value function of infinite-dimensional convex optimization problems under inclusion constraints and of infinite-dimensional convex optimization problems under geometrical and functional constraints. Coderivative of the constraint multifunction, subdifferential, and singular subdifferential of the objective function are the main ingredients in those formulas. 
 
 The present paper discusses differential stability of convex programming problems in Hausdorff locally convex  topological vector spaces. Among other things, we obtain formulas for computing or estimating the subdifferential and the singular subdifferential of the optimal value function via suitable multiplier sets.
 Optimality conditions for convex optimization problems under inclusion constraints and for convex optimization problems under geometrical and functional constraints will be formulated too. But our main aim is to clarify the connection between the subdifferentials of the optimal value function and certain multiplier sets. Namely, by using some results from~\cite{AnYen}, we derive an upper estimate for the subdifferentials via the Lagrange multiplier sets and give an example to show that the upper estimate can be strict. Then, by defining a satisfactory multiplier set, we obtain formulas for computing the subdifferential and the singular subdifferential of the optimal value function. 
 
 As far as we understand, Theorems \ref{computing_subdifferential} and \ref{thm_singular_computing} in this paper have no analogues in the vast literature on differential stability analysis of parametric optimization problems. Here, focusing on convex problems, we are able to give exact formulas for the subdifferential in question under a minimal set of assumptions. It can be added also that the upper estimates in Theorems~\ref{outer_theorem} and~\ref{outer_theorem_singular} are based on that set of assumptions, which is minimal in some sense. 
 
 As one referee of our paper has observed that the results in the convex framework are essentially different from nonconvex ones given, e.g., in the book by Mordukhovich  \cite{Mordukhovich_2006}. The main difference of the results in the present paper and those from \cite{Mordukhovich_Nam2014,Mordukhovich_Nam2015}, and \cite[Theorem~7.2]{Mordukhovich_Nam_Rector_Tran}, is that the latter ones are expressed in terms of the coderivatives of the general convex-graph mappings, while the former ones are given directly via Lagrange multipliers associated with the convex programming constraints. Note that the coderivative calculations for such constraint mappings are presented, e.g., in \cite{Mordukhovich_2006} and the convex
 extremal principle established in \cite[Theorem~2.2]{Mordukhovich_Nam2017} is a main tool of \cite{Mordukhovich_Nam_Rector_Tran}.

 As examples of application of theoretical results on sensitivity analysis (in particular, of exact formulas for computing derivative of the optimal value function) to practical problems, we refer to \cite[Sects.~1 and~6]{Burke}, where the authors considered perturbed linear optimization programs. The results obtained in this paper can be applied to perturbed convex optimization problems in the same manner.
 
The organization of the paper is as follows. Section 2 recalls some definitions from convex analysis, variational analysis, together with several auxiliary results. In Section 3, optimality conditions for convex optimization problems are obtained under suitable regularity conditions. Section 4 establishes formulas for computing and estimating the subdifferential of the optimal value function via multiplier sets. Formulas for computing and estimating the singular subdifferential of that optimal value function are given in Section 5.

\section{Preliminaries}
\markboth{\centerline{\it Preliminaries}}{\centerline{\it D.T.V.~An
and N.D.~Yen}} \setcounter{equation}{0}

Let $X$ and $Y$ be Hausdorff locally convex topological vector spaces with the topological duals denoted, respectively, by $X^*$ and $Y^*$. For a convex set $\Omega\subset X$, the \textit{normal cone} of $\Omega$ at $\bar x\in \Omega$ is given by \begin{align} N(\bar x; \Omega)=\{x^*\in X^* \mid \langle x^*, x-\bar x \rangle \leq 0, \ \, \forall x \in \Omega\}.\end{align} 

Consider a function  $f:X \rightarrow \overline{\mathbb{R}}$ having values in the extended real line $\overline{\mathbb{R}}=[- \infty, + \infty]$. One says that $f$ is \textit{proper} if $f(x) > - \infty$ for all $x \in X$ and if the \textit{domain}	${\rm{dom}}\, f:=\{ x \in X \mid f(x) < +\infty\}$ is nonempty. The set $ {\rm{epi}}\, f:=\{ (x, \alpha) \in X \times \mathbb{R} \mid \alpha \ge f(x)\}$ is called the \textit{epigraph} of $f$. If ${\rm{epi}}\, f$ is a convex (resp., closed) subset of $X \times {\mathbb{R}}$, $f$ is said to be a \textit{convex} (resp., {\it closed}) function.
			\medskip
			
	 The {\it subdifferential} of a proper convex function $f: X\rightarrow \overline{\mathbb{R}}$ at a point $\bar x \in {\rm dom}\,f$ is defined by
		\begin{align}\label{subdifferential_convex_analysis}
		\partial f(\bar x)=\{x^* \in X^* \mid \langle x^*, x- \bar x \rangle \le f(x)-f(\bar x), \ \forall x \in X\}.
		\end{align}	
		
		 The {\it singular subdifferential} of a proper convex function $f: X\rightarrow \overline{\mathbb{R}}$ at a point $\bar x \in {\rm dom}\,f$ is given by
				\begin{align}\label{singular_subdifferential_convex_analysis}
				\partial^\infty f(\bar x)=\{x^* \in X^* \mid (x^*,0)\in N ( (\bar x, f(\bar x));{\rm epi}\,f)\}.
				\end{align}	
		By convention, if $\bar x \notin {\rm dom}\,f$, then $\partial f(\bar x)=\emptyset$ and $\partial^\infty f(\bar x)=\emptyset$.
			
		Note that $x^* \in \partial f(\bar x)$ if and only if $\langle x^*, x-\bar x \rangle -\alpha f(\bar x) \le 0$ for all $(x, \alpha) \in {\rm epi}\, f$
				or, equivalently, $(x^*,-1)\in N( (\bar x, f(\bar x)); {\rm epi}\, f)$. Also, it is easy to show that $\partial\iota_\Omega (x)=N(x;\Omega)$ where $\iota_\Omega (\cdot)$ is the {\it indicator function} of a convex set $\Omega \subset X$. Recall that $\iota_\Omega (x)=0$ if $x \in \Omega$ and $ \iota_\Omega (x)=+\infty$ if $x \notin \Omega$. Interestingly, for any convex function $f$, one has $\partial^\infty f(\bar x)=N(\bar x; {\rm dom}\,f)=\partial  \iota_{{\rm dom}\,f} (\bar x)$~(see \cite[Proposition~4.2]{AnYen}).
	
		One says that a multifunction $F:X\rightrightarrows Y$ is {\it closed} (resp., {\it convex}) if ${\rm gph}\,F$ is a closed (resp., convex) set, where ${\rm gph}\,F:=\{(x,y)\in X \times Y \mid y \in F(x) \}$.			
	
	Given a convex function $\varphi: X \times Y \rightarrow \overline{\mathbb{R}}$, we denote by $\partial_x \varphi(\bar x, \bar y)$ and $\partial_y \varphi(\bar x, \bar y)$, respectively, its \textit{partial subdifferentials} in $x$ and $y$  at $(\bar x,\bar y)$. Thus, $\partial_x \varphi(\bar x, \bar y)=\partial \varphi(., \bar y)(\bar x)$ and $\partial_y \varphi(\bar x, \bar y)=\partial\varphi (\bar x, .)(\bar y)$, provided that the expressions on the right-hand-sides are well defined. It is easy to check that
	\begin{align}\label{inclusion_partial}
	\partial \varphi(\bar x, \bar y) \subset \partial_x \varphi(\bar x, \bar y)\times \partial_y \varphi(\bar x, \bar y).
	\end{align}
		Let us show that inclusion \eqref{inclusion_partial} can be strict. 
	\begin{example}\label{Ex1} {\rm Let $X=Y=\mathbb{R}$, $\varphi(x,y)=|x+y|$, and $\bar x=\bar y=0$. Since  
			$$\varphi(x,y)=|x+y|=\max\{x+y,-x-y\},$$ by applying a well known formula giving an exact expression of the subdifferential of the maximum function \cite[Theorem~3, pp.~201--202]{Ioffe_Tihomirov_1979}
			we get $$\partial \varphi (\bar x, \bar y)={\rm co}\left\{(1,1)^T,\, (-1,-1)^T \right\},$$
			where ${\rm co}\,\Omega$ denotes the \textit{convex hull} of $\Omega$. Since $\partial_x \varphi (\bar x, \bar y)=\partial_y \varphi (\bar x, \bar y)=[-1,1]$, we see that $\partial_x \varphi(\bar x, \bar y)\times \partial_y \varphi(\bar x, \bar y )\not\subset \partial \varphi(\bar x, \bar y).$}
		\end{example}

	In the sequel, we will need the following fundamental calculus rule of convex analysis. 
		
		\begin{theorem} \label{MoreauRockafellar}{\rm (The Moreau-Rockafellar Theorem) (See \cite[Theorem 0.3.3 on pp. 47--50, Theorem 1 on p.~200]{Ioffe_Tihomirov_1979})} Let $f_1,\dots,f_m$ be proper convex functions on $X$. Then
		$$\partial (f_1+\dots+f_m)(x) \supset\partial f_1(x) +\dots+\partial f_m(x)$$ for all $x\in X$.
		If, at a point $x^0\in {\rm dom}\,f_1\cap\dots\cap {\rm dom}\,f_m$, all the functions $f_1,\dots,f_m$, except, possibly, one are continuous, then 
		$$ \partial (f_1+\dots+f_m)(x) =  \partial f_1(x) +\dots + \partial f_m(x)$$ for all $x\in X$.
		\end{theorem}
		
		Another version of the above Moreau-Rockafellar Theorem, which is based on a geometrical regularity condition of Aubin's type, will be used later on. Note that Aubin \cite[Theorem 4.4, p. 67]{Aubin_1998} only proved this result in a Hilbert space setting, but he observed that it is also valid in a reflexive Banach space setting. It turns out that the reflexivity of the Banach space under consideration can be omitted. A detailed proof of the following theorem can be found in~\cite{Bonnans_Shapiro_2000}.
						
		\begin{theorem} \label{Aubin_sumrule}{\rm (See {\cite[Theorem 2.168 and Remark 2.169]{Bonnans_Shapiro_2000}})} Let $X$ be a Banach space. If $f,g: X \rightarrow \overline{\mathbb{R}}$ are proper, closed, convex functions and the regularity condition
				\begin{align}\label{Regularity_condition}
				0 \in\, {\rm{int}} ({\rm dom}\,f- {\rm dom}\,g)
				\end{align}
				holds, then for any $x \in ({\rm dom}\, f)\cap ({\rm dom}\,g)$ we have
				\begin{align}\label{Sum_rule}
				\partial (f + g )(x) =\partial f(x) + \partial g (x),
				\end{align}
				where ${\rm{int}}\,\Omega$ denotes the  interior of a set $\Omega$.
								\end{theorem}
								
 By using the indicator functions of convex sets, one can easily derive from Theorem~\ref{MoreauRockafellar} the next intersection formula. 

\begin{proposition}\label{intersection_formula} {\rm (See \cite[p.~205]{Ioffe_Tihomirov_1979})}  Let $A_1,A_2,\dots,A_m$ be convex subsets of $X$ and  $A=A_1\cap A_2\cap\dots\cap A_m$. Suppose that $A_1\cap ({\rm{int}}\, A_2)\cap\dots\cap  ({\rm{int}}\, A_m) \not= \emptyset$. Then,
$$ N(x;A)=N(x;A_1)+N(x;A_2)+\dots+N(x;A_m), \quad  \forall x\in X.$$
\end{proposition}

The forthcoming theorem characterizes continuity of extended-real-valued convex functions defined on Hausdorff locally convex topological vector spaces.

\begin{theorem}
{\rm(See \cite[p.~170]{Ioffe_Tihomirov_1979})}
Let $f$ be a proper convex function on a Hausdorff locally convex topological vector space $X$. Then the following assertions are equivalent:
\par {\rm{(i)}} $f$ is bounded from above on a neighborhood of a point $x\in X$;

\par{\rm{(ii)}} $f$ is continuous at a point $x\in X$;

\par {\rm{(iii)}} ${\rm{int(epi}}\, f) \not=\emptyset;$

\par {\rm{(iv)}} ${\rm{int(dom}}\, f) \not=\emptyset$ and $f$ is continuous on ${\rm{int(dom}}\, f).$
Moreover,
$$ {\rm{int(epi}}\, f)=\{ (\alpha , x) \in \Bbb{R} \times X \mid x \in {\rm{int(dom}}\, f), \alpha > f(x)\}. $$
\end{theorem}

\medskip

The following infinite-dimensional version of the Farkas lemma \cite[p.~200]{Rockafellar_1970} has been obtained by Bartl \cite{Bartl_2008}.

\begin{lemma}\label{Farkas_lemma}{\rm{(See \cite[Lemma~1]{Bartl_2008})}}
Let $W$ be a vector space over the reals. Let $A: W \rightarrow \Bbb{R}^m$ be a linear mapping and $\gamma : W \rightarrow \Bbb{R}$ be a linear functional. Suppose that $A$ is represented in the form $A=(\alpha_i)_i^m$, where each  $\alpha_i:W\to \Bbb{R}$ is a linear functional (i.e., 
for each $x\in W$, $A(x)$ is a column vector whose $i-th$ component is $\alpha_i(x)$, for $i=1,\dots,m$). Then, the inequality $\gamma(x) \le 0$ is a consequence of the inequalities system
$$\alpha_1(x) \le 0, \ \alpha_2(x) \le 0,\dots, \ \alpha_m(x) \le 0$$
if and only if there exist nonnegative real numbers $\lambda_1, \lambda_2,\dots,\lambda_m \ge 0$ such that
$$\gamma=\lambda_1\alpha_1+\dots+\lambda_m \alpha_m.$$
\end{lemma}

Finally, let us recall a lemma from \cite{AnYen} which describes the normal cone of the intersection of finitely many affine hyperplanes. The proof of this result has been done by applying Lemma \ref{Farkas_lemma}.

\begin{lemma}\label{lemma2} {\rm (See {\cite[Lemma 5.2]{AnYen}})} Let $X, Y$ be Hausdorff locally convex topological vector spaces.
Let there be given vectors $(x_j^*,y_j^*)\in X^*\times Y^*$ and real numbers  $\alpha_j \in \Bbb{R},\ j=1,\dots,k $. Set
$$Q_j=\left\{ (x,y)\in X\times Y \mid \langle (x_j^*,y_j^*), (x,y) \rangle = \alpha_j\right\}.$$ Then, for each  $(\bar x, \bar y)\in  \bigcap\limits_{j=1}^k Q_j$, it holds that
\begin{align}\label{formulanon}
N\left( (\bar x, \bar y); \bigcap\limits_{j=1}^k Q_j\right)={\rm{span}}\{ (x_j^*, y_j^*) \mid j =1,\dots,k \},
\end{align}
where ${\rm{span}}\{ (x_j^*, y_j^*) \mid j =1,\dots,k\}$ denotes the linear subspace generated by the vectors $(x_j^*, y_j^*)$, $j =1,\dots,k$.
\end{lemma}
\section{Optimality conditions}
\markboth{\centerline{\it Optimality conditions}}{\centerline{\it D.T.V.~An
and N.D.~Yen}} 

Optimality conditions for convex optimization problems, which can be derived from the calculus rules of convex analysis, have been presented in many books and research papers. To make our paper self-contained and easy for reading, we are going to present systematically some optimality conditions for convex programs under inclusion constraints and for convex optimization problems under geometrical and functional constraints. Observe that these conditions lead to certain Lagrange multiplier sets which are used in our subsequent differential stability analysis of parametric convex programs.

	Let $X$ and $Y$ be Hausdorff locally convex topological vector spaces. Let $\varphi: X \times Y \rightarrow \overline{\mathbb{R}}$ be a proper convex extended-real-valued function.
 
\subsection{Problems under inclusion constraints}
 Given a convex multifunction $G: X \rightrightarrows Y$, we consider the \textit{parametric convex optimization problem under an inclusion constraint}
			\begin{align*}
 (P_x)	\quad \quad \quad  	\min\{\varphi(x,y)\mid y \in G(x)\} 
		\end{align*}
	 depending on the parameter $x$. The \textit{optimal value function}
	 $\mu: X \rightarrow \overline{\mathbb{R}}$ of problem $(P_x)$ is
	 \begin{align}\label{marginalfunction}
	 \mu(x):= \inf \left\{\varphi (x,y)\mid y \in G(x)\right\}.
	 \end{align}
	 The usual convention $\inf \emptyset =+\infty$ forces $\mu(x)=+\infty$ for every $x \notin {\rm{dom}}\, G.$ 
	 The \textit{solution map}  $M: {\rm {dom}}\, G  \rightrightarrows Y $ of that problem is defined by
		\begin{align}\label{solution_map}
		M(x):=\{y \in G(x)\mid \mu(x)= \varphi (x,y)\}.
		\end{align}
		
The next theorem describes some necessary and sufficient optimality conditions for $(P_x)$ at a given parameter $\bar x \in X$. 

\begin{theorem}\label{MR_version} Let $\bar x \in X$. Suppose that at least one of the following regularity conditions is satisfied:
\par {\rm{(a)}} $ {\rm int }\, G(\bar x) \cap {\rm{dom}}\, \varphi(\bar x, .) \not= \emptyset,$
				\par {\rm{(b)}} $\varphi (\bar x, .)$ is continuous at a point belonging to $G(\bar x).$\\
Then, one has $\bar y\in M(\bar {x})$ if and only if 
\begin{align} \label{NS_condition}
0 \in \partial_y \varphi(\bar x, \bar y) + N(\bar y; G(\bar x)).
\end{align}
\end{theorem}
\begin{proof}
Consider the function $\varphi_G(y)= \varphi (\bar x, y) + \iota_{G(\bar x)}( y)$, where $\iota_{G(\bar x)}(\cdot)$ is the indicator function of the convex set $G(\bar x)$. The latter means that $\iota_{G(\bar x)}(y)=0$ for $y\in G(\bar x)$ and $\iota_{G(\bar x)}(y)=+\infty$ for $y\notin G(\bar x)$. It is clear that $\bar y \in  M(\bar x)$ if and only if the function $\varphi_G$ attains its minimum at $\bar y$. Hence, by \cite[Proposition~1, p. 81]{Ioffe_Tihomirov_1979}, $\bar y \in  M(\bar x)$ if and only if 
\begin{align}\label{NS_condition1}
0 \in \partial \varphi_G(\bar y)=\partial\bigg(\varphi (\bar x, .)+\iota_{G(\bar x)}(.)\bigg )(\bar y).
\end{align}
Since $G(\bar x)$ is convex, $\iota_{G(\bar x)}(\cdot)$ is convex. Clearly, $\iota_{G(\bar x)}(\cdot)$ is continuous at every point belonging to $ {\rm int}\,G(\bar x)$. Thus, if the regularity condition (a) is fulfilled, then $\iota_{G(\bar x)}(\cdot)$ is continuous at a point in ${\rm dom}\, \varphi(\bar x, .)$. By Theorem~\ref{MoreauRockafellar}, from \eqref{NS_condition1} one has
\begin{align*}
0 \in \partial\bigg(\varphi (\bar x, .)+\iota_{G(\bar x)}(.)\bigg)(\bar y)&=\partial_y \varphi(\bar x, \bar y)+ \partial\iota_{G(\bar x)}(\bar y)\\
&=\partial_y \varphi(\bar x, \bar y) + N(\bar y; G(\bar x)).
\end{align*}
Consider the case where (b) holds. Since ${\rm dom}\,\iota_{G(\bar x)}(\cdot)= G(\bar x),$ $\varphi(\bar x,.)$ is continuous at a point in ${\rm dom}\,{\iota_{G(\bar x )}(\cdot)}$. Then, by Theorem \ref{MoreauRockafellar} one can obtain~\eqref{NS_condition} from~\eqref{NS_condition1}.
	$\hfill\Box$
\end{proof}

The sum rule in Theorem \ref{Aubin_sumrule} allows us to get the following result.

\begin{theorem}
Let $X,Y$ be Banach spaces, $\varphi: X \times Y \rightarrow \overline{\mathbb{R}}$ a proper, closed, convex function. Suppose that $G:X \rightrightarrows Y$ is a convex multifunction, whose graph is closed. Let $\bar x \in X$ be such that the regularity condition 
\begin{align}\label{Aubin_RC}
0\in {\rm int}\, \big ({\rm dom}\, \varphi (\bar x,.) - G(\bar x) \big )
\end{align}
is satisfied. Then, $\bar y \in  M(\bar x)$ if and only if 
\begin{align} \label{NS_condition3}
0 \in \partial_y \varphi(\bar x, \bar y) + N(\bar y; G(\bar x)).
\end{align}	
\end{theorem} 
\begin{proof}
The proof is similar to that of Theorem \ref{MR_version}. Namely, if the regularity condition \eqref{Aubin_RC} is fulfilled, then instead of Theorem~\ref{MoreauRockafellar} we can apply Theorem~\ref{Aubin_sumrule} to the case where $X \times Y$, $\varphi(\bar x,.)$, and ${\iota_{G(\bar x)}(\cdot)}$, respectively, play the roles of $X$, $f$, and $g$. 
$\hfill\Box$
\end{proof}

\subsection{Problems under geometrical and functional constraints}
We now study optimality conditions for \textit{convex optimization problems under geometrical and functional constraints}. Consider the program
\begin{align*}(\Tilde{P}_x)	\quad \ \;  \min \left\{\varphi (x,y) \mid y\in C(x),\ g_i(x,y) \le 0, \ i \in I,\  h_j(x,y)=0,\  j\in J \right\} \end{align*}
depending on parameter $x$, where $g_i:X\times Y\to \mathbb{R}$, $i\in I:=\{1,\dots,m\}$, are continuous convex functions, $h_j:X\times Y\to \mathbb{R}$, $j\in J:=\{1,\dots,k\}$, are continuous  affine  functions,  and $C(x):=\{y \in Y: (x,y) \in C \}$ with $C \subset X \times Y$ being a convex set.  For each $x \in X$, we put
\begin{align}\label{constraint}
  G(x)=\left\{y \in Y \mid y\in C(x),\ g(x,y) \le 0, \ h(x,y)=0\right\},
\end{align} 
where $$g(x,y):=(g_1(x,y),\dots,g_m(x,y))^T,\ \; h(x,y):=(h_1(x,y),\dots,h_k(x,y))^T,$$ with $^T$ denoting matrix transposition, and the inequality $z\le w$ between two vectors in $\mathbb{R}^m$ means that every coordinate of $z$ is less than or equal to the corresponding coordinate of $w$. It is easy to show that the multifunction  $G(\cdot)$ given by \eqref{constraint} is convex. Fix a point $\bar x\in X$ and recall that \begin{align}\label{C_x_bar}
C(\bar x)=\{y\in Y\mid (\bar x,y)\in C\}.
\end{align}

 The next lemma describes the normal cone to a sublevel set of a convex function. 
 
 \begin{lemma}\label{lemma1} {\rm (See \cite[Proposition 2 on p.~206]{Ioffe_Tihomirov_1979})}
 Let $f$ be a proper convex function on $X$, which is continuous at a point $x_0\in X$. Assume that the inequality $f(x_1)<f(x_0)=\alpha_0$ holds for some $x_1\in X$. Then, 
 \begin{align}\label{formula_for_normal_cone}
 N(x_0;[f \le \alpha_0])=K_{\partial f(x_0)},
 \end{align}
 where $[f \le \alpha_0]:= \{ x \mid f(x) \le \alpha_0\}$ is a sublevel set of $f$ and 
 $$ K_{\partial f(x_0)}:=\{u^* \in X^* \mid u^*=\lambda x^*,\  \lambda \ge 0,\ x^*\in \partial f(x_0)\}$$ is the cone generated by the subdifferential of $f$ at $x_0$.
 \end{lemma}

Optimality conditions for convex optimization problems under geometrical and functional constraints can be formulated as follows.

\begin{theorem} If $\varphi(\bar x, .)$ is continuous at a point $y^0\in {\rm int}\,C(\bar x)$, $g_i(\bar x,y^0)<0$ for all $i\in I$ and $h_j(\bar x,y^0)=0$ for all $j\in J$, then for a point $\bar y \in G(\bar x)$ to be a solution of $(\Tilde{P}_{\bar x})$, it is necessary and sufficient that there exist $\lambda_i \ge 0,$ $i\in I,$ and $\mu_j \in \Bbb{R},$ $j\in J,$ such that
 \\ {\rm {(a)}}  $0 \in \partial_y \varphi(\bar x,\bar y)+ \sum\limits_{i\in I} \lambda_i \partial_y g_i(\bar x, \bar y)+\sum\limits_{j\in J}\mu_j \partial_y h_j(\bar x, \bar y)+ N(\bar y;C(\bar x)); $
 \\ {\rm {(b)}} $\lambda_ig_i (\bar x, \bar y)=0, \ i\in I.$
 \end{theorem} 
 \begin{proof}
 For any $\bar x \in X$, let $\bar y \in G(\bar x)$ be given arbitrarily. Note that $(\Tilde{P}_{\bar x})$ can be written in the form
 $$\min \big\{\varphi(\bar x,y)\mid y \in G(\bar x)\big\}.$$
 If $\varphi(\bar x, .)$ is continuous at a point $y^0$ with $y^0\in {\rm int}\,C(\bar x)$, $g_i(\bar x,y^0)<0$ for all $i\in I$, and $h_j(\bar x,y^0)=0$ for all $j\in J$, then the regularity condition (b) in Theorem~\ref{MR_version} is satisfied. Consequently, $\bar y\in M(\bar x)$ if and only if 
 \begin{align}\label{Fermat}
 0\in \partial_y \varphi (\bar x, \bar y) +N(\bar y; G(\bar x)).
 \end{align}
 
 We now show that 
 \begin{equation}\label{cone_normal1}
 \begin{split}
 N(\bar y; G(\bar x))=\left\{\!\sum\limits_{i\in I(\bar x, \bar y)} \lambda_i \partial_y g_i(\bar x, \bar y)\!+\!\sum\limits_{j\in J}\mu_j \partial_y h_j(\bar x, \bar y)\!+\! N(\bar y;C(\bar x))\right\},
 \end{split}
 \end{equation}
 with $I(\bar x, \bar y):=\{i \mid g_i(\bar x, \bar y)=0, \, i\in I\},$  $\lambda_i \ge 0,\,  i\in I,\, \mu_j \in \Bbb{R}, \, j\in J.$
 First, observe that
 \begin{align}\label{normal_gph}
 G(\bar x)= \left( \bigcap\limits_{i \in I} \Omega_i(\bar x) \right) \cap \left( \bigcap\limits_{j \in J} \mathcal{Q}_i(\bar x) \right)\cap C,
 \end{align}
 where $\Omega_i(\bar x)= \{ y\mid g_i(\bar x,y) \le 0\} (i \in I)$ and $\mathcal{Q}_j(\bar x)= \{ y\mid h_j(\bar x,y) = 0\} (j \in J)$ are convex sets. By our assumptions, we have 
 \begin{align*}
 y^0\in \left( \bigcap\limits_{i \in I}{\rm int\,} \Omega_i(\bar x) \right) \cap \left( \bigcap\limits_{j \in J} \mathcal{Q}_i(\bar x) \right)\cap ({\rm int}\,C).
 \end{align*}
 Therefore, according to Proposition \ref{intersection_formula} and formula \eqref{normal_gph}, one has 
 \begin{align}\label{Normal_cap}
 N(\bar y;G(\bar x))= \sum\limits_{i \in I} N(  \bar y; \Omega_i(\bar x)) + N\left(  \bar y; \bigcap\limits_{j\in J}\mathcal{Q}_j(\bar x) \right)+N(\bar y;C(\bar x)).
 \end{align}
 On one hand, by Lemma \ref{lemma1}, for every $i \in I(\bar x, \bar y)$ we have 
 \begin{align}\label{normal_Omega}
 N(  \bar y; \Omega_i(\bar x))= K_{\partial_y g_i(\bar x, \bar y)}{=\{\lambda_i y^*\mid  \lambda_i \ge 0,\; y^*\in \partial_y g_i(\bar x, \bar y)\}}.
 \end{align}
 On the other hand, according to Lemma \ref{lemma2} and the fact that {$$h_j(x,y) = \langle x_j^*, x\rangle + \langle y_j^*, y\rangle - \alpha_j\quad \big ((x_j^*,y_j^*)\in X^*\times Y^*,\ \alpha_j\in\mathbb R\big),$$} we can assert that
 \begin{align}\label{normal_Q}
 {N\left(  \bar y; \bigcap\limits_{j\in J}\mathcal{Q}_j(\bar x) \right)}={\rm{span}}\{ y^*_j\mid j\in J  \}={\rm{span}}\{ \partial_y h_j(\bar x, \bar y) \mid j \in J\},
 \end{align}
 Combining \eqref{Normal_cap}, \eqref{normal_Omega}, and \eqref{normal_Q}, we obtain \eqref{cone_normal1}. So the assertion of the theorem is valid. $\hfill\Box$
 \end{proof}
 
 \section{Subdifferential Estimates via Multiplier Sets}
 \markboth{\centerline{\it Subdifferential Estimates via Multiplier Sets}}{\centerline{\it D.T.V.~An
 and N.D.~Yen}} 

 The following result on differential stability of convex optimization problems under geometrical and  functional constraints has been obtained in \cite{AnYen}.
 
 \begin{theorem}\label{Thm5.21} {\rm (See \cite[Theorem~5.2]{AnYen})} For every $j\in J$, suppose that
 	$$ h_j(x,y) = \langle (x_j^*,y_j^*), (x,y) \rangle - \alpha_j, \  \, \alpha_j \in \Bbb{R}.$$ If $\varphi$ is continuous at a point $(x^0,y^0)$ with $(x^0,y^0)\in {\rm int}\,C$, $g_i(x^0,y^0)<0,$ for all $i\in I$ and $h_j(x^0,y^0)=0,$ for all $j\in J$, then for any $\bar x \in {\rm{dom}}\, \mu$, with $\mu(\bar x)\neq -\infty$, and for any $\bar y \in M(\bar x)$ we have
 	\begin{align}\label{UpperEstimate1}
 	\partial \mu(\bar x) = \bigcup\limits_{(x^*,y^*) \in \partial \varphi(\bar x, \bar y)} \big\{x^* + \tilde{Q}^*\big\}
 	\end{align}
 	and
 		\begin{align}\label{UpperEstimate2}
 	 	\partial^\infty \mu(\bar x) = \bigcup\limits_{(x^*,y^*) \in \partial^\infty \varphi(\bar x, \bar y)} \big\{x^* + \tilde{Q}^*\big\},
 	 	\end{align}
 	where
 	\begin{align}\label{Q_star}
 	\tilde{Q}^*&:=\bigg\{u^* \in X^* \mid (u^*,-y^*) \in A+N((\bar x,\bar y);C) \bigg\}
 	\end{align}
 	with 
 	\begin{align}\label{sum_A} A:=\sum\limits_{i\in I(\bar x, \bar y)} {\rm{cone}}\, \partial g_i(\bar x, \bar y) + {\rm{span}}\{(x_j^*,y_j^*), \,j \in J\}.
 	\end{align}
 \end{theorem}
 
 Our aim in this section is to derive formulas for computing or estimating the subdifferential of the optimal value function of $(\tilde{P}_x)$ through suitable multiplier sets.
 
 \medskip
 The \textit{Lagrangian function} corresponding to the parametric problem $(\Tilde{P}_x)$ is
 \begin{align}
 \label{Lagrangian_function}
 L(x,y,\lambda,\mu):= \varphi(x,y) + \lambda^T g(x,y) + \mu^T h(x,y)+{\iota_C((x,y))},
 \end{align}
 where $\lambda=(\lambda_1,\lambda_2,...,\lambda_m)\in \mathbb{R}^m $ and $\mu=(\mu_1,\mu_2,...,\mu_k)\in \mathbb{R}^k.$ For each pair $(x,y)\in X\times Y$, by $\Lambda_0(x,y)$ we denote the set of all the multipliers $\lambda\in \mathbb{R}^m$ and $\mu\in \mathbb{R}^k$ with $\lambda_i\geq 0$ for all $i\in I$ and $\lambda_i=0$ for every $i\in I\setminus I(x,y)$, where $ I(x,y)=\{i\in I\mid g_i(x,y)=0\}$.
 
 \medskip
 For a parameter~$\bar x$, the \textit{Lagrangian function} corresponding to the unperturbed problem  $(\Tilde{P}_{\bar x})$ is 
  \begin{align}
 \label{Lagrangian_function2}
 L(\bar x,y,\lambda,\mu)= \varphi(\bar x, y) + \lambda^T g(\bar x,y) + \mu^T h(\bar x,y)+{\iota_C((\bar x,y))}. 
\end{align} Denote by $\Lambda(\bar x, \bar y)$ the \textit{Lagrange multiplier set} corresponding to an optimal solution $\bar y$ of the problem $(\Tilde{P}_{\bar x})$. Thus, $\Lambda(\bar x, \bar y)$ consists of the pairs $(\lambda, \mu)\in \mathbb{R}^m \times \mathbb{R}^k$ satisfying
    \begin{equation*}\begin{cases}
0 \in \partial _y L(\bar x, \bar y, \lambda, \mu), 
 \\ 
\lambda_i g_i(\bar x, \bar y)=0, \; i=1,\dots,m,\\
 \lambda_i \ge 0,\; i=1,\dots,m,
 \end{cases}
 \end{equation*}
where $\partial_y  L(\bar x, \bar y, \lambda, \mu)$ is the subdifferential of the function $L(\bar x , ., \lambda, \mu)$ defined by \eqref{Lagrangian_function2} at $\bar y$. It is clear that ${\iota_C((\bar x,y))}={\iota_{C(\bar x)}(y)}$, where $C(\bar x)$ has been defined by \eqref{C_x_bar}.

  {Based on the multiplier} set $\Lambda_0(x,y)$, the next theorem provides us with a formula for computing the subdifferential of the optimal value function $\mu(x).$  
    
    \begin{theorem}\label{computing_subdifferential}
Suppose that
   $ h_j(x,y) = \langle (x_j^*,y_j^*), (x,y) \rangle - \alpha_j, \  \alpha_j \in \Bbb{R}, \  j\in J,$ and $M(\bar x)$ is nonempty for some $\bar x \in {\rm dom}\, \mu $. If $\varphi$ is continuous at a point $(x^0,y^0)\in {\rm int}\,C$, $g_i(x^0,y^0)<0$ for all $i\in I$ and $h_j(x^0,y^0)=0$ for all $j\in J$ then, for any $\bar y \in M(\bar x)$, one has    
    	\begin{align}\label{in_estimate}
    	\partial \mu (\bar x)= \left\{ \bigcup\limits_{(\lambda, \mu) \in \Lambda_0(\bar x, \bar y)} {\rm pr}_{X^*}\bigg(\partial L(\bar x, \bar y, \lambda, \mu)\cap \big(X^* \times \{0\}\big)\bigg)\right\},
    	\end{align}
    	where $\partial L(\bar x, \bar y, \lambda, \mu)$ is the subdifferential of the function $L(., ., \lambda, \mu)$ at $(\bar x,\bar y)$ and, for any $(x^*,y^*)\in X^*\times Y^*$, ${\rm pr}_{X^*}(x^*,y^*):=x^*$.
    \end{theorem}
    \begin{proof} {To prove the inclusion ``$\subset$" in \eqref{in_estimate}, take any $\bar x^* \in \partial \mu(\bar x)$}. 
    	By Theorem \ref{Thm5.21},  there exist $(x^*,y^*)\in \partial\varphi (\bar x, \bar y)$ and $u^*\in\tilde{Q}^*$ such that $\bar x^*=x^*+u^*$. According to \eqref{Q_star}, the condition $u^*\in\tilde{Q}^*$ means that \begin{align}\label{basic_inclusion} (u^*,-y^*) \in N((\bar x,\bar y);C)+A,\end{align} {where $A$ is given by \eqref{sum_A}}. Adding the inclusion $(x^*,y^*)\in \partial\varphi (\bar x, \bar y)$ and that one in \eqref{basic_inclusion} yields
    	\begin{align*}
        (x^*+u^*,0) \in (x^*, y^*)+A+N((\bar x,\bar y);C).
    	\end{align*} Hence,
    	\begin{align}\label{incl_for_bar_x_star} (\bar x^*,0)\in \partial\varphi (\bar x, \bar y) +A+N((\bar x,\bar y);C).
    	\end{align} For every $(\lambda, \mu) \in \Lambda_0(\bar x, \bar y)$, the assumptions made on the functions $\varphi$, $g_i$, $h_j$, and the set $C$ allow us to apply the Moreau-Rockafellar Theorem (see Theorem~\ref{MoreauRockafellar}) to the Lagrangian function $L(x,y,\lambda,\mu)$ defined by \eqref{Lagrangian_function} to get 
    	\begin{align}\label{sum_rule_Lagragian1}
    	\partial\, L(\bar x, \bar y, \lambda, \mu)= \partial \varphi(\bar x, \bar y) \!\!+\!\! \sum\limits_{i\in I(\bar x,\bar y)} \lambda_i \partial g_i(\bar x, \bar y)\!+\! \sum\limits_{j\in J}\mu_j \partial h_j(\bar x, \bar y)\!+\!N((\bar x, \bar y);C).
    	\end{align} Since  $\partial h_j(\bar x, \bar y)=\{(x_j^*,y_j^*)\}$, from \eqref{sum_rule_Lagragian1} it follows that
    	\begin{align}\label{basic-equality}\partial\varphi (\bar x, \bar y) +A+N((\bar x,\bar y);C)=\bigcup\limits_{(\lambda, \mu) \in \Lambda_0(\bar x, \bar y)} \partial L(\bar x, \bar y, \lambda, \mu).\end{align}  So, \eqref{incl_for_bar_x_star} means that
       	\begin{align}\label{incl_for_bar_x_star_2} \bar x^*\in\bigcup\limits_{(\lambda, \mu) \in \Lambda_0(\bar x, \bar y)} {\rm pr}_{X^*} \bigg(\partial L(\bar x, \bar y, \lambda, \mu)\cap \big(X^* \times \{0\}\big)\bigg).
    	\end{align}
    	Thus, the inclusion ``$\subset$" in \eqref{in_estimate} is valid. To obtain the reverse inclusion, fixing any $\bar x^*$	satisfying \eqref{incl_for_bar_x_star_2} we have to show that $\bar x^*\in \partial \mu(\bar x)$. As it has been noted before, \eqref{incl_for_bar_x_star_2} is equivalent to \eqref{incl_for_bar_x_star}. Select a pair $(x^*,y^*)\in\partial\varphi (\bar x, \bar y)$ satisfying
    		\begin{align*} (\bar x^*,0)\in (x^*,y^*)+A+N((\bar x,\bar y);C).
    		\end{align*} 
    		Then, for $u^*:=\bar x^*-x^*$, one has 
    			\begin{align*} (x^*+u^*,0)\in (x^*,y^*)+A+N((\bar x,\bar y);C).
    			\end{align*} Therefore, the inclusion \eqref{basic_inclusion} holds. Hence, thanks to \eqref{Q_star} and \eqref{UpperEstimate1}, the vector $\bar x^*=x^*+u^*$ belongs to $\partial \mu(\bar x)$. 
    			
    			The proof is complete. $\hfill\Box$
\end{proof}

{As an illustration for Theorem \ref{computing_subdifferential}, let us consider the following simple example.}
  
  \begin{example}\label{Ex2} {\rm Let $X=Y=\mathbb R$, $C=X\times Y$, $\varphi(x,y)=|x+y|$, $m=1$, $k=0$ (no equality functional constraint), $g_1(x,y)=y$ for all $(x,y)\in X\times Y$. {Choosing $\bar x=0$, one has $M(\bar x)=\{\bar y\}$ with $\bar{y}=0$}. It is clear that  $\Lambda_0(\bar x,\bar y)=[0,\infty)$ and $L(x,y,\lambda)=\varphi(x,y)+\lambda y$. As in Example \ref{Ex1}, we have $$\partial \varphi (\bar x, \bar y)={\rm co}\left\{(1,1)^T,\, (-1,-1)^T \right\}.$$ Since $\partial L(\bar x,\bar y,\lambda)=\partial \varphi (\bar x, \bar y)+\{(0,\lambda)\}$, by \eqref{in_estimate} we can compute \begin{align*}
  		\partial \mu (\bar x)& = \left\{ \bigcup\limits_{\lambda \in \Lambda_0(\bar x, \bar y)} {\rm pr}_{X^*}\bigg(\partial L(\bar x, \bar y, \lambda)\cap \big(X^* \times \{0\}\big)\bigg)\right\}\\
  		& ={\rm pr}_{X^*}\left[\bigg(\bigcup\limits_{\lambda \in \Lambda_0(\bar x, \bar y)} \partial L(\bar x, \bar y, \lambda)\bigg)\cap \big(X^* \times \{0\}\big)\right]\\
  		& = {\rm pr}_{X^*}\bigg\{\left[{\rm co}\left\{(1,1)^T,\, (-1,-1)^T \right\} +\Big(\{0\}\times \mathbb{R}_+\Big)\right]\cap \big(X^* \times \{0\}\big)\bigg\}\\
  		& = [-1,0].
  		\end{align*} To verify this result, observe that
  	\begin{align*}
  	 \mu (x)=\inf\left\{|x+y|\mid y \leq 0\right\}=\begin{cases}
  0, & {\rm if}\ x\geq 0,\\
  -x, & {\rm if}\ x<0.
\end{cases}
  	\end{align*} So we find $\partial \mu (\bar x)= [-1,0]$, justifying \eqref{in_estimate} for the problem under consideration.
}	
  	\end{example}
 
    We are now in a position to establish an upper estimate for the subdifferential $\mu(.)$ at $\bar x$ by using the Lagrange multiplier set $\Lambda(\bar x, \bar y)$ corresponding to a solution $\bar y$ of $(\Tilde{P}_{\bar x})$. 
    
  \begin{theorem}\label{outer_theorem} 
  	Under the assumptions of Theorem \ref{computing_subdifferential}, one has 
   \begin{align}\label{outer_estimate}
   \partial \mu (\bar x) \subset \bigcup\limits_{(\lambda, \mu) \in \Lambda(\bar x, \bar y)} \partial_x L(\bar x, \bar y, \lambda,\mu),
   \end{align}
   where $\partial_x L(\bar x, \bar y, \lambda, \mu)$ stands for the subdifferential of $L(., \bar y, \lambda, \mu) $ at $\bar x$.
    \end{theorem}
     \begin{proof}  Fix an arbitrary vector $\bar {x}^* \in \partial\, \mu(\bar x)$. The arguments in the first part of the proof of Theorem \ref{computing_subdifferential} show that 
     	\eqref{incl_for_bar_x_star} and \eqref{basic-equality} are valid. Hence, we can find a vector $(\lambda, \mu) \in \Lambda_0(\bar x, \bar y)$ such that \begin{align}
     	\label{incl2_for_bar_x_star} (\bar x^*,0)\in  \partial L(\bar x, \bar y, \lambda, \mu).
     	\end{align} Using the definition of subdifferential, from \eqref{incl2_for_bar_x_star} we can deduce that
     	$$\langle \bar x^*,x-\bar x\rangle \leq L(x, \bar y, \lambda, \mu)- L(\bar x, \bar y, \lambda, \mu)\quad \forall x\in X$$
     	and 
     	$$\langle 0,y-\bar y\rangle \leq L(\bar x,y, \lambda, \mu)- L(\bar x, \bar y, \lambda, \mu)\quad \forall y\in Y.$$
     		Hence,
     	 \begin{align}
     	\label{two_incls} 
     \bar x^*\in \partial_xL(\bar x, \bar y, \lambda, \mu),\ \; 
     0\in  \partial_y L(\bar x, \bar y, \lambda, \mu).
     	\end{align} Since $(\lambda, \mu) \in \Lambda_0(\bar x, \bar y)$, one has $\lambda_ig_i(\bar x,\bar y)=0$ and $\lambda_i\geq 0$ for every $i\in I$. Therefore, the second inclusion in \eqref{two_incls} implies that $(\lambda, \mu) \in \Lambda (\bar x, \bar y)$. Then,~\eqref{outer_estimate} follows from  the first inclusion in \eqref{two_incls}.$\hfill\Box$     	
   \end{proof}
   
   The next example shows that the inclusion in Theorem \ref{outer_theorem} can be strict.
   
   \begin{example}{\rm Let $X=Y=\mathbb R$, $C=X\times Y$, $\varphi(x,y)=|x+y|$, $m=1$, $k=0$ (no equality functional constraint), $g_1(x,y)=y$ for all $(x,y)\in X\times Y$. Choosing $\bar x=0$, we note that $M(\bar x)=\{\bar y\}$ with $\bar{y}=0$. We have $L(x,y,\lambda)=\varphi(x,y)+\lambda y$ and \begin{align*}\Lambda(\bar x,\bar y)&=\{\lambda \ge 0\mid 0 \in \partial_y L(\bar x, \bar y, \lambda) \}\\
   &=\{\lambda \ge 0\mid 0 \in [-1,1]+\lambda \}\\
   &=[0,1].
   \end{align*} As in Example \ref{Ex2}, one has $\partial \mu(\bar x)=[-1,0].$ We now compute the right-hand-side of \eqref{outer_estimate}. By simple computation, we obtain $\partial_x L(\bar x, \bar y, \lambda)=[-1,1]$ for all $\lambda \in \Lambda(\bar x, \bar y).$ Then $\bigcup\limits_{\lambda \in \Lambda(\bar x, \bar y)} \partial_x L(\bar x, \bar y, \lambda)=[-1,1].$ Therefore, in this example, inclusion~\eqref{outer_estimate} is strict.
   }
   \end{example}

\section{Computation of the singular subdifferential}
\markboth{\centerline{\it Computation of the singular subdifferential}}{\centerline{\it D.T.V.~An
 and N.D.~Yen}} 
First, we observe that $x\in {\rm dom}\,\mu$ if and only if 
$$\mu(x)={\rm inf}\{\varphi(x,y)\mid y \in G(x) \}< \infty,$$ with $G(x)$ being given by \eqref{constraint}.
Since the strict inequality holds if and only if {there exists} $y \in G(x)$ with $(x,y)\in {\rm dom}\, \varphi$, we have 
\begin{align}
\label{new_problem}
{\iota_{{\rm dom}\,\mu}(x)}={\rm inf}\{{\iota_{{\rm dom}\, \varphi}((x,y))}\mid y \in G(x) \}.
\end{align}
To compute the singular subdifferential of $\mu(.)$, let us consider the minimization problem 
\begin{align*}\big(P^\infty_x\big)\quad \begin{cases} {\iota_{{\rm dom}\, \varphi}((x,y))}\to \inf & \\
{\rm subject\ to}\ \; y\in C(x),\  g_i(x,y) \le 0,\  i \in I,\ h_j(x,y)=0,\ j\in J.&
\end{cases}\end{align*}

The Lagrangian function corresponding to $(P^\infty_x)$ is
\begin{align}
\label{new_largrange_function}
 \widehat{L}(x,y,\lambda,\mu):= {\iota_{{\rm dom}\, \varphi}((x,y))} + \lambda^T g(x,y) + \mu^T h(x,y)+{\iota_C((x,y))},
 \end{align}
 where $\lambda=(\lambda_1,\lambda_2,...,\lambda_m)\in \mathbb{R}^m,$  $\mu=(\mu_1,\mu_2,...,\mu_k)\in \mathbb{R}^k$. 
 
 \medskip
 Interpreting $(P^\infty_x)$ as a problem of the form $(\widetilde P_x)$, where ${\iota_{{\rm dom}\, \varphi}((x,y))}$ plays the role of $\varphi(x,y)$, we can apply Theorem~\ref{computing_subdifferential} (resp., Theorem~\ref{outer_theorem}) to compute (resp., estimate) the singular subdifferential of $\mu(.)$ as follows.

\begin{theorem}\label{thm_singular_computing}
Under the hypotheses of Theorem \ref{computing_subdifferential}, for any $\bar y \in M(\bar x)$, one has     
 	    	\begin{align}\label{singular_computing}
 	    	\partial^\infty \mu (\bar x)= \left\{ \bigcup\limits_{(\lambda, \mu) \in \Lambda_0(\bar x, \bar y)} {\rm pr}_{X^*} \bigg(\partial\widehat{L}(\bar x, \bar y, \lambda, \mu)\cap \big(X^* \times \{0\}\big)\bigg)\right\},
 	    	\end{align}
 	    	where \begin{align}\label{singular_computing2}\partial \widehat{L}(\bar x, \bar y, \lambda, \mu)\!=\!\partial^\infty\varphi(\bar x, \bar y) \!+\! \sum\limits_{i\in I(\bar x,\bar y)} \lambda_i \partial g_i(\bar x, \bar y)\!+\! \sum\limits_{j\in J}\mu_j \partial h_j(\bar x, \bar y)\!+\!N((\bar x, \bar y);C)\end{align} is the subdifferential of the function $\widehat{L}(., ., \lambda, \mu) $ at $(\bar x,\bar y)$, provided that a pair $(\lambda, \mu) \in \Lambda_0(\bar x, \bar y)$ has been chosen.
 	    \end{theorem} 
 	    \begin{proof} The inclusion $\bar y \in M(\bar x)$ implies that $(\bar x,\bar y)\in {\rm dom}\, \varphi$ and $\bar y \in G(\bar x)$. So, ${\iota_{{\rm dom}\, \varphi}((\bar x,\bar y))}=0$ and $\bar y$ is a feasible point of the problem $\big(P^\infty_{\bar x}\big)$. As ${\iota_{{\rm dom}\, \varphi}((\bar x,y))}\geq 0$ for all $y\in G(\bar x)$, we can  assert that $\bar y$ is a solution of~$\big(P^\infty_{\bar x}\big)$. The corresponding optimal value is ${\iota_{{\rm dom}\,\mu}(\bar x)}=0$ (see \eqref{new_problem}). Hence, by Theorem \ref{computing_subdifferential} and formula \eqref{new_problem}, we have 
 	    		\begin{align*} 	\partial{\iota_{ {\rm dom}\,\mu}(\bar x)}
 	    	= \left\{ \bigcup\limits_{(\lambda, \mu) \in \Lambda_0(\bar x, \bar y)} {\rm pr}_{X^*} \bigg(\partial\widehat{L}(\bar x, \bar y, \lambda, \mu)\cap \big(X^* \times \{0\}\big)\bigg)\right\}.
 	    	\end{align*} Since $\partial{\iota_{ {\rm dom}\,\mu}(\bar x)}=\partial^\infty \mu (\bar x)$, the last equality implies \eqref{singular_computing}. 
 	    	
 	    	For every $(\lambda, \mu) \in \Lambda_0(\bar x, \bar y)$, remembering that $h_j$, $j\in J$, are affine functions, $\varphi$ is continuous at a point $(x^0,y^0)$ with $(x^0,y^0)\in {\rm int}\,C$, $g_i(x^0,y^0)<0$ for all $i\in I$ and $h_j(x^0,y^0)=0$ for all $j\in J$, we can apply Theorem~\ref{MoreauRockafellar} to the Lagrangian function ${\widehat L}(x,y,\lambda,\mu)$ defined by \eqref{new_largrange_function} to obtain
 	    	\begin{align*}\begin{array}{rl}
 	    	& \partial \widehat{L}(\bar x, \bar y, \lambda, \mu)\\
 	    	& =\partial{\iota_{{\rm dom}\, \varphi}((\bar x,\bar y))}  \!\!+\!\! \sum\limits_{i\in I(\bar x,\bar y)} \lambda_i \partial g_i(\bar x, \bar y)\!+\! \sum\limits_{j\in J}\mu_j \partial h_j(\bar x, \bar y)\!+\!N((\bar x, \bar y);C).
 	    	\end{array}\end{align*} Combining this with the equality $\partial{\iota_{{\rm dom}\, \varphi}((\bar x,\bar y))}=\partial^\infty\varphi(\bar x, \bar y)$ yields \eqref{singular_computing2}.
 	      $\hfill\Box$	
 	    \end{proof}
 	    
 	    \begin{remark} {\rm The result in Theorem \ref{thm_singular_computing} can be derived from formula \eqref{UpperEstimate2} by a proof analogous to that of Theorem  \ref{computing_subdifferential}.}
 	    	\end{remark}
     	
     	Next, denote by $\Lambda^\infty (\bar x, \bar y)$ the \textit{singular Lagrange multiplier set} corresponding to an optimal solution $\bar y$ of the problem $(P^\infty_{\bar x})$, which consists of the pairs $(\lambda, \mu)\in \mathbb{R}^m \times \mathbb{R}^k$ satisfying
     	\begin{equation*}\begin{cases}
     	0 \in \partial _y {\widehat L}(\bar x, \bar y, \lambda, \mu), 
     	\\ 
     	\lambda_i g_i(\bar x, \bar y)=0, \; i=1,\dots,m,\\
     	\lambda_i \ge 0,\; i=1,\dots,m.
     	\end{cases}
     	\end{equation*}
     Here $\partial_y {\widehat L}(\bar x, \bar y, \lambda, \mu)$ is the subdifferential of the function ${\widehat L}(\bar x , ., \lambda, \mu)$, with ${\widehat L}(x,y, \lambda, \mu)$ being given by \eqref{new_largrange_function}, at $\bar y$.
     	
     	 \begin{theorem}\label{outer_theorem_singular} 
     		Under the assumptions of Theorem \ref{computing_subdifferential}, for any $\bar y \in M(\bar x)$, one has 
     		\begin{align}\label{outer_estimate_singular}
     		\partial^\infty \mu (\bar x) \subset \bigcup\limits_{(\lambda, \mu) \in \Lambda^\infty (\bar x, \bar y)} \partial_x {\widehat L}(\bar x, \bar y, \lambda,\mu),
     		\end{align}
     		where $\partial_x {\widehat L}(\bar x, \bar y, \lambda, \mu)$ stands for the subdifferential of ${\widehat L}(., \bar y, \lambda, \mu) $ at $\bar x$. 
     	\end{theorem}
     	\begin{proof} To get \eqref{outer_estimate_singular}, it suffices to apply Theorem~\ref{outer_theorem} to the parametric problem~$(P^\infty_{\bar x})$, keeping in mind that $\bar y$ is a solution of $(P^\infty_{\bar x})$. Indeed, taking account of Theorem~\ref{outer_theorem} and \eqref{outer_estimate}, one has
     		\begin{align*} \partial{\iota_{{\rm dom}\,\mu}(\bar x)}
     		\subset \bigcup\limits_{(\lambda, \mu) \in \Lambda^\infty (\bar x, \bar y)} \partial_x {\widehat L}(\bar x, \bar y, \lambda,\mu).
     		\end{align*} As $\partial{\iota_{ {\rm dom}\,\mu}(\bar x)}=\partial^\infty \mu (\bar x)$, this inclusion is equivalent to \eqref{outer_estimate_singular}.    $\hfill\Box$	
     	\end{proof}

	\section*{Acknowledgements}
	
\noindent 
	The first author was supported by the Vietnam Institute for Advanced Study in Mathematics (VIASM) and Thai Nguyen University of Sciences. This research is funded by the National Foundation for Science and Technology Development (Vietnam) under grant number 101.01-2014.37. The authors would like to thank the two anonymous referees for their very careful readings and valuable suggestions which have helped to greatly improve the presentation.

\end{document}